\documentclass[11pt]{article}

\usepackage{amsmath, amsfonts, amssymb}
\usepackage{mathrsfs}
\usepackage{theorem}
\usepackage{bm}

\newtheorem{thm}{Theorem}[section]
\newtheorem{prop}[thm]{Proposition}
\newtheorem{lem}[thm]{Lemma}

{\theorembodyfont{\rmfamily} \newtheorem{rem}[thm]{Remark}}
{\theorembodyfont{\rmfamily} }

\newcommand{\ra}{\rightarrow}
\newcommand{\dis}{\displaystyle}

\def\R{\mathbb R}

\def\d{\text{\rm{d}}}
\def\E{\mathbb E}
\def\p{\mathbb P}\def\e{\text{\rm{e}}}

\def\La{\Lambda}
\def\veps{\varepsilon}

\def\S{\mathcal S}

\newcommand{\fin}{\hspace*{\fill}\rule{0.3em}{1ex}}
\newenvironment{proof}{{\bf \noindent Proof.}}{\fin}

\numberwithin{equation}{section}

\begin{document}

\title{Permanence and extinction of regime-switching predator-prey models
}

\author{
  Jianhai Bao${}^b$,  Jinghai Shao${}^a$\\[0.2cm]
{\small a: School of Mathematical Sciences, Beijing Normal University, Beijing 100875, China}\\
{\small b: Department of Mathematics, Central South University,
Changsha, Hunan 410083,   China}
}
\maketitle
\begin{abstract}
  In this work we study the permanence and extinction of  a regime-switching predator-prey model
  with Beddington-DeAngelis functional response. The switching process is used to describe
  the random changes of corresponding parameters such as birth and death rates of a species in different environments. When a prey will die out in
  some fixed environments and will not in the others, our criteria can justify whether it dies out or not in a random switching environment.
  Our criteria are rather sharp, and they cover the known on-off type results on permanence of predator-prey models without switching. Our method relies on the recent study of ergodicity of regime-switching diffusion processes.
\end{abstract}
\section{Introduction}
In an ecosystem, all species  evolve and compete to seek resources
to sustain their existence. Denote the two population sizes at time
$t$ by $X_t$ and $Y_t$, respectively. $X_t$ denotes the population
size of the prey and $Y_t$ stands for the population size of the
predator. The Kolmogorov predator-prey model is a general
deterministic model taking the form:
\begin{equation*}
  \begin{cases}
    \dot X_t=X_tf(X_t,Y_t),\\
    \dot Y_t=Y_tg(X_t,Y_t).
  \end{cases}
\end{equation*}
For $f(x,y)=b-py$ and $g(x,y)=cx-d$, one gets the well-known
Lotka-Volterra model. These deterministic models have been studied
extensively. We refer to the monograph \cite{Sigm} due to Hofbauer
and Sigmund for a related study on this deterministic model. Here
$f(X,Y)$ and $g(X,Y)$ stand for the capita growth rate of each
species, which is dependent on the population sizes of both species.
Various functional response functions have been used considerably in
modeling population dynamics such as Holling types, Hassel-Varley
type, Leslie-Gower Holling  II type, Beddington-DeAngelis type, etc.
However, due to the continuous fluctuation in the environment, the
birth rates, death rates, carrying capacities, competition
coefficients and all the other parameters involved in this model
exhibit random fluctuation to a great extent. To describe this
phenomenon, stochastic predator-prey models with different kinds of
responses are proposed and there are many works to study these
models (cf. \cite{AHS}, \cite{JJ}, \cite{KK} and references
therein).
 For a predator-prey system with
Beddington-DeAngelis functional response,
 in practice, we
usually estimate the birth rate of the prey and death rate of the
predator by their average values plus errors which follow normal
distributions. 
As a result,
one gets the following stochastic predator-prey model
   with Beddington-DeAngelis functional response:
\begin{equation}\label{1.1}
\begin{cases}
\d X_t=X_t\Big(a_1-b_1X_t-\frac{c_1 Y_t}{m_1+m_2 X_t+m_3Y_t}\Big)\d t+\alpha X_t\d B_1(t),\\[0.2cm]
\d Y_t=  Y_t\Big(-a_2-b_2 Y_t+\frac{c_2X_t}{m_1+m_2X_t+m_3Y_t}\Big)\d t+\beta Y_t\d B_2(t),
\end{cases}
\end{equation}
with $X_0=x_0>0$ and $Y_0=y_0>0$. All the parameters in (\ref{1.1})
are positive, and $(B_1(t))$, $(B_2(t))$ are  Brownian motions on
the line. For the model \eqref{1.1}, by Khasminskii's criterion on
existence of stationary distributions for diffusion processes (see,
e.g., \cite[Theorem 4.1]{Kha}), Ji-Jiang \cite{JJ} revealed that
$(X_t,Y_t)$ enjoys a stationary distribution $\mu$, which is also
ergodic;
 Du et al. \cite{DDY} showed that if $-a_2-\frac{\beta^2}{2}-\int_{-\infty}^{\infty} \frac{c_2
\e^x}{m_1+m_2\e^x} f^\ast(x)\d x=:\lambda<0$, where $f^\ast(x)=
C\exp(\frac{2a_1-\alpha^2}{\alpha^2}x-\frac{2b_1}{\alpha^2} \e^x)$
with $C$ being the normalizing constant,   the predator $Y_t$ will
eventually extinct (i.e. $\lim_{t\rightarrow\infty}Y_t=0$ a.s.),
and that, in the case $\lambda>0$,  $(X_t,Y_t)$ has a unique
invariant probability measure $\mu^*$ which is supported in
$\R_+^{2,o}$; Liu-Wang \cite{LW} showed that the equilibrium
$(x^*,y^*)$ is stochastically asymptotically stable in the large
(i.e. for any initial data $(x_0,y_0)$,
$\lim_{t\rightarrow\infty}X_t=x^*$ a.s. and
$\lim_{t\rightarrow\infty}Y_t=y^*$ a.s.); Li-Zhang \cite{LZ}
provided some sufficient conditions for non-persistence  in the mean
(i.e. $\limsup_{t\rightarrow\infty} (\frac{1}{t}\int_0^tX_s\d s
)=0$ a.s.), strong persistence in the mean (i.e.
$\liminf_{t\rightarrow\infty} (\frac{1}{t}\int_0^tX_s\d s )>0$
a.s.), and weak persistence in the mean (i.e.
$\limsup_{t\rightarrow\infty} (\frac{1}{t}\int_0^tX_s\d s )>0$ a.s.)
of the prey population $X_t$.

From the viewpoint of biological modeling, variability of the
environment  may have an important impact on the dynamics  of the
community.  For instance, the distinctive season  change such as dry
and rainy seasons is observed in monsoon forests, and it
characterizes the vegetation there. Also, in boreal and arctic
regions, seasonality exerts a strong influence on the dynamics of
mammals.  Moreover, the growth rates, the death rates and the
carrying capacities often vary according to the changes in nutrition
and food resources. All of these changes usually cannot be described
by the traditional deterministic or stochastic predator-prey models.
Therefore, it is natural to consider the predator-prey model in a
random environment, which is formulated by an additional factor
process.  More precisely, consider the following regime-switching
predator-prey model with Beddington-DeAngelis functional response:
\begin{equation}\label{1.2}
\begin{cases}
\d X_t=X_t\Big(a_1(\La_t)-b_1(\La_t)X_t-\frac{c_1(\La_t) Y_t}{m_1(\La_t)+m_2(\La_t)X_t+m_3(\La_t)Y_t}\Big)\d t+\alpha(\La_t)X_t\d B_1(t)\\[0.2cm]
\d Y_t=Y_t\Big(-a_2(\La_t)-b_2(\La_t)Y_t+
\frac{c_2(\La_t)X_t}{m_1(\La_t)+m_2(\La_t)X_t+m_3(\La_t)Y_t}\Big)\d
t+\beta(\La_t)Y_t\d B_2(t)
\end{cases}
\end{equation}
with $X_0=x_0>0$ and $Y_0=y_0>0$, where $(B_1(t))$ and $(B_2(t))$
are Brownian motions on the line, and $(\La_t)$ is a continuous time
Markov chain with a finite state space $\S=\{1,2,\ldots,N\}$, $1\leq
N<\infty$. Throughout this paper, the processes $(B_1(t))$,
$(B_2(t))$ and $(\La_t)$ are defined on a
  complete probability space $(\Omega,\mathscr F, \{\mathscr F_t\}_{t\geq 0}, \mathbb P)$, and $(\La_t)$ is independent of $(B_1(t))$ and $(B_2(t))$.
 The
parameters $a_k(\cdot)$, $b_k(\cdot)$, $c_k(\cdot)$ for $k=1,2$, and
$m_l(\cdot)$ for $l=1,2,3$, are all positive functions on $\S$.

The dynamical system \eqref{1.2} is a  regime-switching diffusion
process, which has been widely applied in control problems, storage
modeling, neuronal activity, biology and mathematical finance. We
refer the readers to \cite{BBG, CH, PS, Sh14a, Sh1, SX, SX2} and the
monographs \cite{MY, YZ} for the study on recurrence, ergodicity,
stochastic stability, numerical approximation of regime-switching
diffusion processes with Markovian switching or state-dependent
switching in a finite state space or an infinite state space. There
is a vast literature on  population dynamics with regime switching.
For instance, Du-Du \cite{DD} described the omega-limit set of
Kolmogorov systems of competitive type under the telegraph noise and
investigated properties of stationary density; Zhu-Yin \cite{ZYb}
examined certain long-run-average limits, and Zhu-Yin \cite{ZY09}
investigated long-time behavior of sample paths for competitive
Lotka-Volterra ecosystems. In the study of a population system,
permanence and extinction are two important and interesting
properties, respectively  meaning that the population system will
survive or die out in the future. Yuan et al. \cite{YML} discussed
extinction for stochastic hybrid delay population dynamics with $n$
interacting species,  and Li et al. \cite{LAJM} and Liu-Wang
\cite{LWa} studied permanence and extinction for stochastic logistic
populations with single species.

The regime-switching predator-prey model can describe a very
important and interesting situation. We consider a simple example to
introduce it. Let us consider the case $\S=\{1,2\}$, where ``1"
denotes the rainy season and ``2" denotes the dry season. The
situation that $a_1(1)-\frac 12\alpha^2(1)>0$ and $a_1(2)-\frac
12\alpha^2(2)<0$ is rather possible to occur. This means that the
birth rate with perturbation of $(X_t)$ in the rainy season makes
sure that $(X_t)$ will not die out, but that of $(X_t)$ in the dry
season makes $(X_t)$ die out. Then a natural question arises: will
$(X_t)$ in model \eqref{1.2} die out or not? This question is very
interesting and represents an essential advantage of model
\eqref{1.2} over model \eqref{1.1}. However, the solution of this
question is not easy to derive, and so far there are few results of
this type on the regime-switching predator-prey model or
regime-switching population dynamics. On this topic, we refer the
reader to \cite{PS} for the explicit examples to see the complexity
of the regime-switching diffusion processes, and to \cite{CH, Sh14a,
Sh14b, SX2, YZ} for some solutions of this type on ergodicity and
stability of regime-switching diffusion processes. In this work, we
shall provide a sharp criterion to justify whether $(X_t)$ will die
out or not (see Theorem \ref{main} below).

Let $(q_{ij})_{i,j\in\S}$ be the $Q$-matrix of the Markov chain
$(\La_t)$ which means that
\begin{equation} \label{1.3}
\p(\La_{t+\delta}=l|\La_t=k)=\left\{\begin{array}{ll} q_{kl} \delta+o(\delta), &\text{if}\ k\neq l,\\
                                       1+q_{kk} \delta+o(\delta), & \text{if}\ k=l,
                         \end{array}\right.
\end{equation}
for a sufficiently small $\delta>0$. Throughout this work, the
matrix $Q=(q_{ij})$ is assumed to be irreducible and conservative,
i.e. $q_{kk}=-q_k:=-\sum_{j\neq k} q_{kj}<0$. As $\S$ is a finite
set and $(q_{ij})$ is irreducible, the theory of Markov chains tells
us that $(\La_t)$ is ergodic and there exists a unique stationary
distribution $(\mu_i)$ for it. To state our main result, we need to
introduce two auxiliary processes. Let
\begin{equation}
  \label{phi}
  \d \varphi_t=\varphi_t(a_1(\La_t)-b_1(\La_t)\varphi_t)\d t+\alpha(\La_t)\varphi_t\d B_1(t),
\end{equation}
and
\begin{equation}
  \label{psi}
  \d \psi_t=\psi_t\Big(-a_2(\La_t)+\frac{c_2(\La_t)}{m_2(\La_t)}-b_2(\La_t)\psi_t\Big)\d t+\beta(\La_t)\psi_t\d
  B_2(t)
\end{equation} with $\varphi_0=X_0=x_0>0$ and $\psi_0=Y_0=y_0>0$, where $(\La_t)$,  $(B_1(t))$, $(B_2(t))$ are defined as in (\ref{1.2}).
By the comparison theorem for SDEs (cf. \cite{IW}), $\varphi_t\geq
X_t$ and $\psi_t\geq Y_t$ a.s. for all $t>0$. Our main result of
this work is the following theorem.
\begin{thm}\label{main}
Let $(X_t,Y_t,\La_t)$ be defined by (\ref{1.2}) and (\ref{1.3}) and $(\mu_i)_{i\in\S}$ be the stationary distribution  of the process  $(\La_t)$.
\begin{itemize}
  \item[$(i)$] If \,$\dis\sum_{i\in\S} \mu_i\big(a_1(i)-\frac 12 \alpha^2(i)\big)<0$, then $\dis\lim_{t\ra \infty} X_t=0$ a.s.
  , $\dis\lim_{t\ra \infty} Y_t=0$ a.s.
  \item[$(ii)$] If\, $\dis \sum_{i\in \S}\mu_i\big(a_1(i)-\frac 12 \alpha^2(i)\big)>0$, then $(\varphi_t,\La_t)$ is positive recurrent
   with stationary distribution $\pi^\varphi$ on $\R_+\times \S$. Assume
   further that
  \begin{equation}\label{lam}
  \lambda:=-\sum_{i\in\S}\mu_i\big(a_2(i)+\frac 12\beta^2(i)\big)+\sum_{i\in\S}\int_{\R^+}\frac{c_2(i)x}{m_1(i)+m_2(i)x}\pi^\varphi(\d
  x,i)<0.
  \end{equation}
  Then, $\dis\lim_{t\ra \infty} Y_t=0$ a.s., $\dis\limsup_{t\ra\infty}X_t>0$ a.s., and the distribution of $(X_t,\La_t)$ converges weakly to $\pi^\varphi$.
  \item[$(iii)$] If\, $\dis\sum_{i\in \S} \mu_i\Big(a_2(i)+\frac 12 \beta^2(i)-\frac{c_2(i)}{m_2(i)}\Big)<0$, then $(\psi_t,\La_t)$ is positive recurrent with stationary distribution $\pi^\psi$.
   Assume further that $\dis \sum_{i\in \S}\mu_i\big(a_1(i)-\frac 12 \alpha^2(i)\big)>0$ and
  \begin{equation}\label{barlam}
  \bar\lambda:=\lambda+\sum_{i\in\S}\mu_i\Big(a_2(i)+\frac 12\beta^2(i)-\frac{c_2(i)}{m_2(i)}\Big)+\sum_{i\in \S}\int_{\R_+}b_2(i)y\pi^\psi(\d
  y,i)>0.
  \end{equation}
  Then, $\dis\limsup_{t\ra \infty} X_t>0$ a.s., $\dis\limsup_{t\ra \infty} Y_t>0$ a.s., and $(X_t,Y_t,\La_t)$ has a stationary distribution.
\end{itemize}
\end{thm}

This theorem will be proved in the next section. As we mentioned,
there are few results on permanence for the model \eqref{1.2}
although there have been numerous works on extinction. Whereas,  in
Theorem \ref{main}, we provide a criterion which can justify whether
a prey dies out or not when it lives in a random switching
environment such that it will die out in some fixed environments and
won't die out in others. We shall note that, when $\S$ contains only
one state, and hence there is no switching in (\ref{1.2}) in this
case,  our result will coincide  with  the results in \cite{DDY}.
Actually, according to a similar calculation of \cite[(2.3)]{DDY},
$\sum_{i\in\S}\mu_i\big(a_2(i)+\frac
12\beta^2(i)-\frac{c_2(i)}{m_2(i)}\big)+\sum_{i\in
\S}\int_{\R_+}b_2(i)y\pi^\psi(\d y,i)=0$ in the case $N=1$, hence
$\bar\lambda=\lambda$ in this case. But when $N>1$, we have no way
to calculate the precise value of this term at present stage due to
lack of explicit representation of the invariant measures of
regime-switching diffusion processes. Moreover, we should point out
that \cite{LZ} investigated strong (weak) permanence {\it in the
mean} for the model \eqref{1.1}. Compared with \cite{DDY}, the main
difficulty in the present work is to determine the recurrence
property of stochastic processes $(\varphi_t,\La_t)$ and
$(\psi_t,\La_t)$, which is overcome by using the recent results in
\cite{Sh14b} on justifying the recurrence property of
regime-switching diffusion processes.

The distributions $\pi^\varphi$ and $\pi^\psi$ are stationary distributions of one-dimensional regime-switching
diffusion processes. In \cite{Sh1}, for one-dimensional regime-switching diffusion process, an explicit representation
of the stationary distribution is provided based on the hitting times of this process.

At last, according to the argument of our main result,  the
permanence or non-permanence of $(X_t, Y_t)$ does not depend on the
correlation of the Brownian motions $(B_1(t))$ and $(B_2(t))$. 
  So in the
 present work, we do not impose
any condition on the dependence between $(B_1(t))$ and $(B_2(t))$.

\section{Proof of the main result}
We first investigate the properties of the processes $(\varphi_t,\La_t)$ and $(\psi_t,\La_t)$ including the estimate of their
 moments and recurrence property.
Set
\begin{equation}
  \label{notation}
  \hat a_1:=\min_{i\in \S} a_1(i) ,\ \check{a}_1:=\max_{i\in\S}a_1(i)
\end{equation}
and similarly 
define $\hat a_2$, $\check a_2$, $\hat b_k$,
$\check b_k$, $\hat c_k$, $\check c_k$ for $k=1,2$ and $\hat m_l$,
$\check m_l$ for $l=1,2,3$. By the finiteness of $\S$, all the
parameters introduced here remain positive.

\begin{lem}\label{t-1}
For any $p>1$,
\begin{equation}\label{2.2}
\begin{split}
  \E\varphi_t^p&\leq \Big[(\E\varphi_0^p)^{-\frac 1p}\exp\Big(-\Big(\check
  a_1+\frac{p-1}{2}\check{\alpha}^2\Big)t\Big)\\
 &\quad +\frac{\hat b_1}{\check a_1+\frac{p-1}{2}\check{\alpha}^2}\Big(1-\exp\Big(-\Big(\check a_1+\frac{p-1}{2}\check\alpha^2\Big)t\Big)\Big)\Big]^{-p},
\end{split}
\end{equation}
and
\begin{equation}\label{2.3}
\begin{split}
  \E\psi_t^p&\leq \Big[(\E\psi_0^p)^{-\frac 1p}\exp\Big(-\Big(-\hat a_2+\frac{\check c_2}{\hat m_2}+\frac{p-1}{2}\check \beta^2\Big)t\Big)\\
  &\quad +\frac{\hat b_2}{-\hat a_2+\frac{\check c_2}{\hat m_2} +\frac{p-1}{2}\check \beta^2}\Big(1-\exp\Big(-\Big(-\hat a_2+\frac{\check c_2}{\hat m_2}
  +\frac{p-1}{2}\check \beta^2\Big)t\Big)\Big)\Big]^{-p}.
\end{split}
\end{equation}
Therefore, for any $p>1$,
\begin{equation}\label{2.4}
  \limsup_{t\ra\infty} \E\varphi_t^p\leq \Big(\frac{\check a_1+\frac{p-1}{2}\check\alpha^2}{\hat b_1}\Big)^{p},
\end{equation}
and for any $p>1$ such that $ -\hat a_2+\frac{\check c_2}{\hat m_2}+\frac{p-1}{2}\check\beta^2>0$, it holds
\begin{equation}\label{2.5}
  \limsup_{t\ra\infty}\E\psi_t^p\leq \Big(\frac{-\hat a_2+\frac{\check c_2}{\hat m_2}+\frac{p-1}{2}\check \beta^2}{\hat b_2}\Big)^p.
\end{equation}
\end{lem}

\begin{proof}
  We shall only prove the estimate for $\varphi_t$ since the estimate for $\psi_t$ can be done in the same way.
  By It\^o's formula,
  \begin{align*}
    \d \varphi_t^p&=p\Big(a_1(\La_t)+\frac{p-1}{2}\alpha^2(\La_t)\Big)\varphi_t^p\d t-pb_1(\La_t)\varphi_t^{p+1}\d t+p\alpha(\La_t)\varphi_t^p\d B_1(t)\\
    &\leq p\Big(\check a_1+\frac{p-1}2\check\alpha^2\Big)\varphi_t^p\d t-p\hat b_1\varphi_t^{p+1}\d t+p\alpha(\La_t)\varphi_t^p\d B_1(t).
  \end{align*}
  Taking the expectation on both sides and utilizing H\"older's inequality yields that
  \begin{equation*}
  \begin{split}
      \frac{\d \E\varphi_t^p}{\d t}&\leq p\Big(\check a_1+\frac{p-1}{2}\check \alpha^2\Big)\E\varphi_t^p-p\hat b_1\E\varphi_t^{p+1}\\
      &\leq p\Big(\check a_1+\frac{p-1}{2}\check \alpha^2\Big)\E\varphi_t^p-p\hat b_1\big(\E\varphi_t^p\big)^{\frac{p+1}{p}}.
    \end{split}
  \end{equation*}
  By the comparison theorem for ordinary  differential equations, we get (\ref{2.2}), and then (\ref{2.4}) follows immediately.
\end{proof}

Note that \eqref{2.5} holds for $p$ sufficiently large. By Lemma \ref{t-1}, the processes $(\varphi_t,\La_t)$ and $(\psi_t,\La_t)$ are
 nonexplosive, and hence the processes $(X_t,\La_t)$ and $(Y_t,\La_t)$ are also nonexplosive 
  due to $\varphi_t\geq
X_t$ a.s. and $\psi_t\geq Y_t$ a.s. for all $t>0$.

\begin{lem}\label{t-0}
It holds that
\begin{equation}\label{positive}
\p(X_t>0,\,\forall\, t>0)=1\quad \text{and}\quad \p(Y_t>0,\,\forall\,t>0)=1.
\end{equation}
\end{lem}
\begin{proof}
  Here we only consider the process $(X_t)$ while the result for $(Y_t)$ can be dealt with in a similar manner.
  Let \[\tau_\Delta=\inf\{t>0;\ X_t\leq \Delta\},\ \Delta>0; \quad \sigma_K=\inf\{t>0;\ X_t\geq K\},\ \ K\geq 1.\]
  Let $\tau_0=\inf\{t>0;\ X_t=0\}$, then $\tau_{\Delta}\uparrow \tau_0$ as $\Delta\downarrow 0$.
  As $(X_t,\La_t)$ is nonexplosive, $\lim_{K\ra \infty}\sigma_K=\infty$ a.s. By It\^o's formula, for $p>0$ such
  that $-\hat a_1+\frac{\check c_1}{\hat m_3}+\frac{p+1}{2}\check
  \alpha^2>0$, we have
  \begin{align*}
    \E X_{t\wedge \sigma_K\wedge\tau_{\Delta}}^{-p}&=x_0^{-p}+p\E\int_0^{t\wedge \sigma_K\wedge \tau_{\Delta}}\!\!\Big(-a_1(\La_s)+\frac{p+1}{2}\alpha^2(\La_s)\\&
    \qquad\qquad\qquad+\frac{c_1(\La_s)Y_s}{m_1(\La_s)
    +m_2(\La_s)X_s+m_3(\La_s)Y_s}+b_1(\La_s)X_s\Big)X_s^{-p} \d s\\
    &\leq x_0^{-p}+p\E\int_0^{t\wedge \sigma_K\wedge\tau_{\Delta}}\!\!\Big(-\hat a_1+\frac{\check c_1}{\hat m_3}
    +\frac{p+1}{2}\check \alpha^2+\check b_1 K\Big)X_s^{-p}\d s.
  \end{align*}
  By Gronwall's inequality,
  \begin{equation}\label{2.0.1}
  \E X_{t\wedge \sigma_K\wedge\tau_{\Delta}}^{-p}\leq x_0^{-p}\exp\Big(p\Big(-\hat a_1 +\frac{\check c_1}
  {\hat m_3}+\frac{p+1}2\check\alpha^2+\check b_1K\Big)t\Big).
  \end{equation}
  If $\p(\tau_0<\infty)>0$,  we can choose $t,\,K$ large enough so that $\p(\tau_0<t\wedge\sigma_K)>0$.
  Then for any $\Delta>0$,
  \begin{align*}
    0&<\p(\tau_0<t\wedge \sigma_K)\leq \p(\tau_\Delta<t\wedge \sigma_K)\\
    &\leq \Delta^p\E[X_{t\wedge \sigma_K\wedge\tau_{\Delta}}^{-p}\mathbf 1_{\tau_\Delta<t\wedge \sigma_K}]\leq \Delta^p\E[X_{t\wedge \sigma_K\wedge\tau_{\Delta}}^{-p}]\\
    &\leq \Delta^px_0^{-p}\exp\Big(p\Big(-\hat a_1 +\frac{\check c_1}
  {\hat m_3}+\frac{p+1}2\check\alpha^2+\check b_1K\Big)t\Big).
  \end{align*}
  Taking $\Delta\downarrow 0$, we get a contradiction. Hence $\p(\tau_0<\infty)=0$ and we complete the proof.
\end{proof}

Next, we go to study the recurrence property of the processes
$(\varphi_t,\La_t)$ and $(\psi_t,\La_t)$, which plays a fundamental
role in the proof of our Theorem \ref{main}.

\begin{lem}[Key lemma]\label{t-2}
\begin{itemize}
\item[$\mathrm{(i)}$]
When $\sum_{i\in \S}\mu_i(a_1(i)-\frac 12 \alpha^2(i))>0$,   the process $(\varphi_t,\La_t)$ is positive recurrent and has a unique stationary distribution $\pi^\varphi$, which is a probability measure on $(0,\infty)\times \S$. When $\sum_{i\in\S}\mu_i(a_1(i)-\frac 12\alpha^2(i))<0$, the process $(\varphi_t,\La_t)$ is transient.

\item[$\mathrm{(ii)}$] When  $\sum_{i\in \S}\mu_i(a_2(i)+\frac 12\beta^2(i)-\frac{c_2(i)}{m_2(i)})<0$,   the process $(\psi_t,\La_t)$ is positive recurrent and has a unique stationary distribution $\pi^\psi$, which is a probability measure on $(0,\infty)\times \S$.
    When $\sum_{i\in\S}\mu_i(a_2(i)+\frac 12\beta^2(i)-\frac{c_2(i)}{m_2(i)})>0$, the process $(\psi_t,\La_t)$ is transient.
\end{itemize}
\end{lem}

\begin{proof}
  (1) We shall apply the method adopted in the argument of \cite[Theorem 3.1]{Sh14b} to prove this lemma. As the diffusion coefficient of $\varphi_t$ is degenerate at
  the point $x=0$, we use the transform $Z_t=\ln \varphi_t$. By Lemma \ref{t-0}, this transform makes sense for all $t\geq 0$
  a.s.
  Then $Z_t$ satisfies the following SDE:
  \[\d Z_t=(a_1(\La_t)-\frac 12 \alpha^2(\La_t)-b_1(\La_t)\e^{Z_t})\d t+\alpha(\La_t)\d B_1(t).\]
  The recurrence property of $(\varphi_t, \La_t)$ is equivalent to that of $(Z_t,\La_t)$.  For each $i\in \S$, set
  $L^{(i)}:=(a_1(i)-\frac12 \alpha^2(i)-b_1(i) \e^x)\frac{\d}{\d x}+\frac12 \alpha^2(i)\frac{\d^2}{\d x^2}$.
  Then the operator $\mathscr A$ defined by
  \[\mathscr A f(x,i)=L^{(i)} f(\cdot,i)(x)+\sum_{j\neq i} q_{ij}(f(x,j)-f(x,i)),\quad f\in C^2(\R\times \S)\]
  is the infinitesimal generator of the process $(Z_t,\La_t)$.

  Set $\beta_i:=a_1(i)-\frac 12 \alpha^2(i)$, $i\in\S$. We first show the positive recurrence of $(Z_t,\La_t)$.
  Take $h(x)=x^2$ and $g(x)=|x|$. Then
  \begin{equation}\label{Lya-1}
  L^{(i)} h(x)=2\big(\beta_i-b_1(i)\e^x+\frac 12 \alpha^2(i)x^{-1}\big)x,\quad \text{for $|x|>0$}.
  \end{equation}
  Note that $\lim_{x\ra +\infty} -b_1(i) \e^x=-\infty$, and  $\lim_{x\ra -\infty} -b_1(i)\e^x=0$.
  Due to $\sum_{i\in\S}\mu_i\beta_i>0$, there exist  $\veps>0$ and  $r_0>0$ such that $\sum_{i\in\S}\mu_i(\beta_i-\veps)>0$ and
  \begin{equation}\label{2.2.1}
   2(\beta_i-b_1(i) \e^x+\frac 12\alpha^2(i)x^{-1})x\leq -2(\beta_i-\veps)g(x),\quad \text{for $x>r_0$},\ i\in \S.
  \end{equation}
  \begin{equation}\label{2.2.-1}
  2(\beta_i-b_1(i)\e^x+\frac 12 \alpha^2(i)x^{-1})x\leq -2(\beta_i-\veps)g(x),\quad \text{for $x<-r_0$}, \ i\in\S.
  \end{equation}
 Thanks to $- \sum_{i\in\S}\mu_i(\beta_i-\veps)<0$, by the Fredholm alternative (cf. \cite[pp.434]{PP}), there exist a constant $\kappa>0$ and a vector $\bm{\xi}$ such that
  \[Q\bm{\xi}(i)=-\kappa+2(\beta_i-\veps),\ \ i\in\S.\]
  Set $V(x,i):=h(x)+\xi_i g(x)$. We have
  \begin{align*}
    \mathscr A V(x,i)&=L^{(i)} h(x)+\xi_i L^{(i)} g(x)+Q\bm{\xi}(i) g(x)\\
    &\leq \Big(-2(\beta_i-\veps) +Q\bm{\xi}(i)+\frac{L^{(i)} g(x)}{g(x)}\xi_i\Big) g(x)\\
    &=\Big(-\kappa+\frac{L^{(i)} g(x)}{g(x)}\xi_i\Big)g(x),\quad |x|>r_0.
  \end{align*}
  As $\lim_{x\ra+\infty}\frac{L^{(i)}g(x)}{g(x)}=-\infty$, $\lim_{x\ra -\infty}\frac{L^{(i)}g(x)}{g(x)}=0$, and $\lim_{|x|\ra +\infty}g(x)=+\infty$, there exists a constant $r_1>r_0$ such that
  \[\mathscr A V(x,i)<-1,\quad |x|>r_1,\ i\in\S.\]
  Consequently, by taking $\lim_{|x|\ra \infty} h(x)=\infty$ into consideration,
  $(Z_t,\La_t)$  and hence $(\varphi_t,\La_t)$ are positive recurrent. By
  \cite[Theorem 4.3, pp.114]{YZ},
 $(Z_t,\La_t)$ has a unique stationary distribution   on $\R\times \S$. So the stationary distribution $\pi^\varphi$ of $(\varphi_t,\La_t)$ is a probability measure on $(0,\infty)\times \S$.

  Next, we prove the transience of $(Z_t,\La_t)$ under the condition that $\sum_{i\in\S}\mu_i\beta_i<0$. Since the behavior of $Z_t$, caused by
  the term $-b_1(i)e^x$, is quite different on $(0,+\infty)$ and
  $(-\infty,0)$,
  we prove
   the transience property  by it's definition.

  Let $h(x)=\frac{1}{|x|}$ and $g(x)=h'(x)$ for $|x|>0$.
  Then
  \begin{align*}
    L^{(i)} h(x)=\big(\beta_i-b_1(i)\e^x-\frac{\alpha^2(i)}{x}\big)\frac1{x^2},\quad\text{for}\ x<0,\ i\in\S.
  \end{align*}
In view of
  \[\lim_{x\ra -\infty}\Big( \beta_i-b_1(i)\e^x-\frac{\alpha^2(i)}{x}\Big)=\beta_i,\]
  there exist $\veps>0$, $r_1>0$ so that
  \[\sum_{i\in\S}\mu_i(\beta_i+\veps)<0,\quad \text{and}\ \beta_i-b_1(i)\e^x-\frac{\alpha^2(i)}{x}<\beta_i+\veps,\ \forall\,x<-r_1, \ i\in\S.
  \]
  Owing to  $\sum_{i\in\S} \mu_i(\beta_i+\veps)<0$, the Fredholm alternative yields that there exist a constant $\kappa>0$ and a vector $\bm\xi$ such that \[Q\bm\xi(i)=-\kappa-\beta_i-\veps.\]
  Set $V(x,i):=h(x)+\xi_ig(x)$ for $x<0$. We derive that
  \begin{align*}
    \mathscr A V(x,i)&=L^{(i)} h(x)+\xi_iL^{(i)} g(x)+Q\bm\xi(i)g(x)\\
    &\leq \Big(\beta_i+\veps+\xi_i\frac{L^{(i)} g(x)}{g(x)}+Q\bm\xi(i)\Big) g(x)\\
    &=\Big(-\kappa+\xi_i\frac{L^{(i)} g(x)}{g(x)}\Big) g(x),\quad \text{for}\ x<-r_1,\ i\in \S.
  \end{align*}
  It is easy to see that $\lim_{x\ra -\infty} \frac{L^{(i)} g(x)}{g(x)}=0$. Denote by $\xi_{\min}:=\min\{\xi_i;\ i \in \S\}$ and $\xi_{\max}:=\max\{\xi_i;\ i\in \S\}$. Hence, there exists a constant $r_2$ with $r_2>r_1>0$ such that  $h(x)+\xi_{\min}g(x)$ is an increasing function on $(-\infty, -r_2]$ and
  \[\mathscr A V(x,i)\leq 0,\quad \text{for}\ x\leq -r_2,\ i\in \S.\]
  Now, take $Z_0=z<-r_2<0$ such that
  \begin{equation}\label{2.2.3}h(z)+\xi_{\max}g(z)<h(-r_2)+\xi_{\min}g(-r_2).\end{equation} Set
  \[\tau_{-K}:=\inf\{t>0; Z_t=-K\},\quad \tau:=\inf\{t>0; Z_t=-r_2\}.\]
  By Dynkin's formula, we have
  \begin{align*}
    \E[V(Z_{t\wedge\tau_{-K}\wedge\tau},\La_{t\wedge\tau_{-K}\wedge\tau})]&=V(z,i_0)
    +\E\int_0^{t\wedge\tau_{-K}\wedge\tau}\mathscr AV(Z_s,\La_s)\d s\\
    &\leq V(z,i_0),
  \end{align*} where  $\La_0=i_0$.
  Letting $t\ra \infty$, we get
  \[\E[V(-K,\La_{\tau_{-K}})\mathbf 1_{\tau_{-K}\leq \tau}]+\E[V(-r_2,\La_{\tau})\mathbf 1_{\tau<\tau_{-K}}]\leq V(z,i_0),\]
  which further yields  that
  \begin{gather*}
    (h(-K)+\xi_{\min}g(-K))\p(\tau\geq \tau_{-K})+(h(-r_2)+\xi_{\min}g(-r_2))\p(\tau<\tau_{-K})\leq V(z,i_0),\\
    \p(\tau\geq \tau_{-K})\geq \frac{h(z)+\xi_{\max}g(z)-h(-r_2)-\xi_{\min}g(-r_2)}
    {h(-K)+\xi_{\min}g(-K)-h(-r_2)-\xi_{\min}g(-r_2)}>0,
  \end{gather*}
  where in the last step we have used the increasing property of  $x\mapsto h(x)+\xi_{\min} g(x)$ on $(-\infty,-r_2]$ and \eqref{2.2.3}.
  Due to Lemma \ref{t-0}, and $\varphi_t\geq X_t $ a.s., we obtain $\tau_{-K}\ra \infty$ a.s. as $K\ra \infty$. Consequently, by passing to the limit as $K\ra\infty$, the previous inequality yields
  \[\p(\tau=\infty)>0,\]
  which implies that the process $(Z_t,\La_t)$ is transient, hence $(\varphi_t,\La_t)$ is transient too.

  (2) The results of $(\psi_t,\La_t)$ can be proved by the method analogous to that used above for the process $(\psi_t,\La_t)$, which is omitted.
\end{proof}

\begin{prop}[Strong ergodicity theorem]\label{t-3}
 Assume that $(\varphi_t,\La_t)$ and $(\psi_t,\La_t)$ are positive recurrent. Let $f$ be a bounded measurable function on $\R\times \S$. Then
 almost surely
 \begin{equation}\label{erg-1}
 \lim_{t\ra \infty}\frac 1t\int_0^tf(\varphi_s,\La_s)\d s=\sum_{i\in\S}\int_{\R_+}\!f(x,i)\pi^\varphi(\d x,i),
 \end{equation}
 and
 \begin{equation}\label{erg-2}
 \lim_{t\ra \infty}\frac 1t\int_0^tf(\psi_s,\La_s)\d s=\sum_{i\in\S}\int_{\R_+}\!f(y,i)\pi^\psi(\d y,i).
 \end{equation}
 Moreover,
 \begin{equation}\label{erg-3}
 \liminf_{t\ra \infty} \frac 1t \int_0^t b_1(\La_s)\varphi_s\,\d s\geq \sum_{i\in\S}\int_{\R_+}\!b_1(i)x\,\pi^\varphi(\d x,i), \quad a.s.,
 \end{equation}
 and \begin{equation}\label{erg-4}
 \liminf_{t\ra \infty}\frac 1t\int_0^tb_2(\La_s)\psi_s\,\d s\geq \sum_{i\in\S}\int_{\R_+}\!b_2(i)y\,\pi^\psi(\d y,i),\quad a.s.
 \end{equation}
\end{prop}

\begin{proof}
Similar to \cite[Theorem 4.4]{YZ}, one can use the idea of \cite[Theorem 3.16, pp.46]{S89} to prove \eqref{erg-1} and \eqref{erg-2} for regime-switching diffusion processes.

  Applying (\ref{erg-1}) to the function $f(z,k)=b_1(k)\min\{z,M\}$ for a positive constant $M$, we get
  \[\lim_{t\ra \infty}\frac 1t \int_0^tb_1(\La_s)\min\{\varphi_s,M\}\d s=\sum_{k\in \S}\!\int_{\R_+}\!b_1(k)\min\{z,M\}\pi^\varphi(\d z,k) \quad a.s.\]
  Then
  \begin{align*}
    \liminf_{t\ra \infty}\frac 1t\int_0^tb_1(\La_s)\varphi_s\d s&\geq \lim_{t\ra\infty}\frac 1t\int_0^tb_1(\La_s)\min\{\varphi_s,M\}\d s\\
    &=\sum_{k\in\S}\!\int_{\R_+}\!b_1(k)\min\{z,M\}\pi^\varphi(\d z,k)\quad a.s.
  \end{align*}
  Letting $M$ tend to $\infty$ in the previous inequality, we obtain (\ref{erg-3}). The proof is   completed.
\end{proof}

After the preparations of the above results on the auxiliary
processes $(\varphi_t)$ and $(\psi_t)$, we are ready to prove our
main result. As the proof of Theorem \ref{main} is a little bit
long, we divide it into three propositions.
\begin{prop}\label{t-4}
If\, $\dis\sum_{i\in\S}\mu_i(a_1(i)-\frac 12\alpha^2(i))<0$, then almost surely $\dis\lim_{t\ra\infty} X_t=0$  and $\dis\lim_{t\ra \infty} Y_t=0$.
\end{prop}
\begin{proof}
  By It\^o's formula, we have
  \begin{align*}
    \d \ln X_t&=\Big(a_1(\La_t)-\frac 12 \alpha^2(\La_t)-b_1(\La_t)X_t-\frac{c_1(\La_t) Y_t}{m_1(\La_t)+m_2(\La_t)X_t+m_3(\La_t)Y_t}\Big)\d t+\alpha(\La_t)\d B_1(t)\\
    &\leq \big(a_1(\La_t)-\frac12 \alpha^2(\La_t)\big)\d t+\alpha(\La_t)\d B_1(t)=:\d \ln \tilde{\varphi}_t.
  \end{align*}
  By the comparison theorem for stochastic differential equations (cf. \cite{IW}), we have \[\ln X_t\leq \ln \tilde{\varphi}_t\ \ \text{for any $t>0$}\ a.s.\]
  Consequently,
  \[\frac{\ln X_t }{t} \leq \frac{\ln \tilde\varphi_t}{t} =\frac{\ln x_0}{t}+\frac 1t\int_0^t\big(a_1(\La_s)
  -\frac12\alpha^2(\La_s)\big)\d s+\frac{1}t\int_0^t\alpha(\La_s)\d B_1(s).\]
  By the strong ergodicity theorem,
  \[\lim_{t\ra \infty} \Big\{\frac 1t\int_0^t\big(a_1(\La_s)-\frac 12 \alpha^2(\La_s)\big)\d s+
  \frac{1}t\int_0^t\alpha(\La_s)\d B_1(s)\Big\}=\sum_{i\in \S}\mu_i(a_1(i)-\frac 12\alpha^2(i))\quad a.s.\]
 Here we have used the fact that
  \[\lim_{t\ra \infty}\frac 1t\int_0^t \!\alpha(\La_s)\d B_1(s)=0\quad a.s.,\]
  which is due to the boundedness of $(\alpha(k))$ and the strong law of large numbers (cf. \cite[Theorem 1.3.4]{Mao}).
  Therefore, when $\sum_{i\in \S}\mu_i(a_1(i)-\frac 12\alpha^2(i))<0$, we have
  $\limsup_{t\ra \infty}\frac{\ln X_t}{t}<0$ a.s., which further implies that  $\dis\lim_{t\ra \infty} X_t=0$ a.s.

  With $\lim_{t\ra \infty} X_t=0$ a.s. in hand, we shall show that  $\lim_{t\ra\infty} Y_t=0$ a.s.  Namely, when the prey
   $(X_t)$ dies out, the predator $(Y_t)$ also dies out.
  By It\^o's formula, it holds
  \begin{align*}
   \limsup_{t\ra \infty} \frac{\ln Y_t}{t}&\leq \limsup_{t\ra \infty}\Big[\frac{\ln y_0}{t}\!+\!\frac 1t\int_0^t\!\Big(\!-\!a_2(\La_s)\!-\!\frac{\beta^2(\La_s)}{2}\!
   -\!\frac{c_2(\La_s)X_s}{m_1(\La_s)}\Big)\d s+\frac{1}{t}\int_0^t\!\beta(\La_s)\d B_2(s)\Big]\\
    &=-\sum_{i\in\S}\mu_i(a_2(i)+\frac12 \beta^2(i))<0,\quad a.s.,
  \end{align*} which yields immediately that $\lim_{t\ra \infty}Y_t=0$ a.s. The proof is therefore completed.
\end{proof}

\begin{prop}\label{t-5}
If\, $\dis \sum_{i\in \S}\mu_i\big(a_1(i)-\frac 12 \alpha^2(i)\big)>0$, and $\lambda<0$, where $\lambda$ is defined by (\ref{lam}), then almost surely $\lim_{t\ra \infty} Y_t=0$,   $\limsup_{t\ra \infty} X_t>0$, and the distribution of $(X_t,\La_t)$ converges weakly to $\pi^\varphi$.
\end{prop}

\begin{proof}
  We first show that $(Y_t)$ tends to be extinct. By It\^o's formula, we have
  \[\d \ln Y_t=\Big(-a_2(\La_t)-\frac12\beta^2(\La_t)-b_2(\La_t)Y_t
  +\frac{c_2(\La_t)X_t}{m_1(\La_t)+m_2(\La_t)X_t+m_3(\La_t)Y_t}\Big)\d t+\beta(\La_t)\d B_2(t).\]
  Since $X_t\leq \varphi_t$ a.s. for any $t\geq 0$ and $(\varphi_t,\La_t)$ is positive recurrent thanks to Lemma \ref{t-2}, in addition to
   the comparison theorem and Proposition \ref{t-3}, we obtain that
  \begin{align*}
    \limsup_{t\ra\infty}\frac 1t \ln Y_t&\leq \limsup_{t\ra \infty}\frac 1t\int_0^t\Big(-a_2(\La_s)-\frac1 2\beta^2(\La_s)+\frac{c_2(\La_s)\varphi_s}{m_1(\La_s)+m_2(\La_s)\varphi_s}\Big)\d s\\
    &\quad +\limsup_{t\ra \infty}\frac 1t\int_0^t\!\beta(\La_s)\d B_2(s)\\&=\lambda<0,\quad a.s.
  \end{align*}
  Therefore, $\lim_{t\ra \infty} Y_t=0$ a.s.

  Next, as $\lim_{t\ra \infty} Y_t=0$ a.s., for any $\veps>0$ there exist  a measurable subset $\Omega_\veps\subset \Omega$ with $\p(\Omega_\veps)>1-\veps$ and a positive constant $t(\veps)$ such that for any $t>t(\veps)$, $\omega\in \Omega_\veps$,
  \begin{align*}
    \d X_t&\geq X_t\Big(a_1(\La_t)-b_1(\La_t)X_t-
    \frac{c_1(\La_t)\veps}{m_1(\La_t)+m_2(\La_t)X_t+m_3(\La_t)\veps}\Big)\d t+\alpha(\La_t)X_t\d
    B_1(t),
  \end{align*}
  which yields that the distribution of $(X_t,\La_t)$ converges weakly to the stationary distribution $\pi^\varphi$ of $(\varphi_t,\La_t)$ as
  $t\ra \infty$. Indeed, let $(\varphi^\veps_t)$ satisfy
  \[\d \varphi_t^\veps=\varphi_t^\veps\Big(a_1(\La_t)-b_1(\La_t)\varphi_t^\veps-
  \frac{c_1(\La_t)\veps}{m_1(\La_t)+m_2(\La_t)\varphi_t^\veps+m_3(\La_t)\veps}\Big)\d t
  +\alpha(\La_t)\varphi_t^\veps \d B_1(t),\]
  with $ \varphi_0^\veps=x_0$. For any bounded increasing continuous functions $f$ on $\R$, we have
  \[\E f(\varphi_t)\geq \E f(X_t)\geq \E[f(X_t)\mathbf 1_{\Omega_\veps}]-\veps \|f\|\geq \E[f(\varphi_t^\veps)\mathbf 1_{\Omega_\veps}]-\veps\|f\|,\]
  where $\|f\|=\sup_x|f(x)|<\infty$. Letting $\veps\ra 0$ and then $t\ra \infty$, we get
  \[\lim_{t\ra \infty} \E f(X_t)=\sum_{i\in\S}\int_{\R_+}f(x)\pi^\varphi(\d x,i).\]
  The linearity in the previous equation implies that it also holds for every bounded decreasing continuous functions, hence for every bounded variation functions. Then it is easy to see that the distribution of $(X_t,\La_t)$ converges weakly to $\pi^\varphi$ as desired.

  To complete the proof, we also need to show that $\limsup_{t\ra \infty} X_t>0$
  a.s.
Indeed, if $\p(\lim_{t\ra \infty}X_t=0)>0$, there exists a measurable subset $\Omega_0$ of $\Omega$ such that $\p(\Omega_0)>0$ and $\forall\,\omega\in\Omega_0$, $\lim_{t\ra\infty} X_t(\omega)=0$.
  For any $\delta>0$, as the distribution of $(X_t,\La_t)$ converges weakly to $\pi^\varphi$, we have
  \begin{align*}
    \pi^\varphi([0,\delta]\times \S)\geq \limsup_{t\ra\infty}\p(X_t\in [0,\delta])\geq \p(\Omega_0)>0.
  \end{align*}
  By the arbitrariness of $\delta$, we have $\pi^\varphi(\{0\}\times\S)>\p(\Omega_0)>0$ which contradicts the fact that $\pi^\varphi$ is a probability measure on $(0,\infty)\times\S$ (see Lemma \ref{t-2}).
  We get the desired conclusion.
\end{proof}

\begin{prop}\label{t-6}
If\, $\dis\sum_{i\in \S}\mu_i(a_1(i)-\frac1 2\alpha^2(i))>0$, $\dis\sum_{i\in\S}\mu_i\big(a_2(i)+\frac 12\beta^2(i)-\frac{c_2(i)}{m_2(i)}\big)<0$, and $\bar\lambda>0$, where $\bar\lambda$ is defined by (\ref{barlam}), then  $\dis\limsup_{t\ra \infty} X_t>0$ a.s.,
$\dis \limsup_{t\ra\infty} Y_t>0$ a.s. and $(X_t, Y_t,\La_t)$ has a stationary distribution.
\end{prop}

\begin{proof}
  Since almost surely $X_t\leq \varphi_t$ and $Y_t\leq \psi_t$ for any $t>0$, the strong ergodicity theorem yields that almost surely
  \begin{equation}\label{2.6.1}
    \limsup_{t\ra\infty}\frac 1t \ln X_t\leq \limsup_{t\ra \infty} \frac 1t\ln\varphi_t\leq\sum_{i\in\S}\mu_i\big(a_1(i)-\frac 12\alpha^2(i)\big)
    -\sum_{i\in\S}\int_{\R_+}\!b_1(i) x\pi^\varphi(\d x,i),\\
\end{equation}
    \begin{equation}
    \label{2.6.2}
    \begin{split}
    \limsup_{t\ra \infty}\frac 1t\ln Y_t&\leq \limsup_{t\ra \infty}\frac1 t\ln\psi_t
    \leq\!-\!\sum_{i\in\S}\mu_i\big(a_2(i)\!+\!\frac 12\beta^2(i)\!-\!\frac{c_2(i)}{m_2(i)}\big)
    \!\\
    &\quad-\!\sum_{i\in\S}\!\int_{\R_+}\!b_2(i)y\pi^\psi(\d y,i).
\end{split}
  \end{equation}
  Next, we apply the trick used in \cite[Theorem 2.2]{DDY} to estimate the lower bounds.
  \begin{align*}
    \frac 1t\ln X_t&=\frac 1t \ln x_0+\frac 1t\int_0^t\!(a_1(\La_s)-\frac 12\alpha^2(\La_s)-b_1(\La_s)\varphi_s)\d s\\
    &\quad+\frac 1t\int_0^t\!\Big(\!b_1(\La_s)(\varphi_s-X_s)-\frac{c_1(\La_s)Y_s}{m_1(\La_s)
    \!+\!m_2(\La_s)X_s\!+\!m_3(\La_s)Y_s}\Big)\d s +\frac 1 t\int_0^t\!\alpha(\La_s)\d B_1(s)\\
    &\geq \frac 1t \ln x_0+\frac 1t\int_0^t\big(a_1(\La_s)-\frac1 2\alpha^2(\La_s)-b_1(\La_s)\varphi_s\big)\d s\\
    &\quad+\frac1 t\int_0^t\!\Big(b_1(\La_s)(\varphi_s-X_s)-\frac{c_1(\La_s)Y_s}{m_1(\La_s)}\Big)\d s+\frac 1t \int_0^t\!\alpha(\La_s)\d B_1(s),\quad a.s.
  \end{align*}
  Combining this with (\ref{2.6.1}), we get
  \begin{equation*}
   \liminf_{t\ra \infty} \frac  1t \int_0^t\!\Big(b_1(\La_s)(X_s-\varphi_s)+\frac{c_1(\La_s)Y_s}{m_1(\La_s)}\Big)\d s
   \geq 0, \quad a.s.,
  \end{equation*}
  and further
  \begin{equation} \label{2.6.3}
  \liminf_{t\ra \infty}\frac 1t\int_0^t\!\big(\hat b_1(X_s-\varphi_s)+\frac{\check c_1}{\hat m_1}Y_s\big)\d s\geq 0,\quad a.s.
  \end{equation}
  For the process $(Y_t)$, we have
\begin{align*}
  \frac 1t \ln Y_t&=\frac 1t \ln y_0-\frac 1t\int_0^t\!b_2(\La_s)Y_s\d s
  -\frac 1t\int_0^t\!\Big(a_2(\La_s)+\frac 12\beta^2(\La_s)-\frac{c_2(\La_s)\varphi_s}{m_1(\La_s)+m_2(\La_s)\varphi_s}\Big)\d s\\
  &\quad -\frac 1t\int_0^t\!\Big(\frac{c_2(\La_s)\varphi_s}{m_1(\La_s)+m_2(\La_s)\varphi_s}
  -\frac{c_2(\La_s)X_s}{m_1(\La_s)+m_2(\La_s)X_s}\Big)\d s\\
  &\quad -\frac1t\int_0^t\!\Big(\frac{c_2(\La_s)X_s}{m_1(\La_s)+m_2(\La_s)X_s}
  -\frac{c_2(\La_s)X_s}{m_1(\La_s)+m_2(\La_s)X_s+m_3(\La_s)Y_s}\Big)\d s\\
  &\quad +\frac 1t \int_0^t\beta(\La_s)\d B_2(s)\\
  &\geq \frac 1t\ln y_0+\frac 1t\int_0^t\!\Big(-a_2(\La_s)-\frac 12\beta^2(\La_s)+\frac{c_2(\La_s)\varphi_s}{m_1(\La_s)+m_2(\La_s)\varphi_s}\Big)\d s\\
  &\quad -\frac 1t\int_0^t\!\Big(\frac{c_2(\La_s)(\varphi_s-X_s)}{m_1(\La_s)+m_2(\La_s)\varphi_s}
  +\Big(\frac{c_2(\La_s)m_3(\La_s)}{m_1(\La_s)m_2(\La_s)}+b_2(\La_s)\Big)Y_s\Big)\d s\\
  &\quad+\frac 1t\int_0^t\beta(\La_s)\d B_2(s), \quad a.s.
\end{align*}
It follows that
\begin{equation}\label{2.6.4}
\begin{split}
&\liminf_{t\ra \infty}\frac 1t \int_0^t\!\Big(\frac{c_2(\La_s)(\varphi_s-X_s)}{m_1(\La_s)+m_2(\La_s)\varphi_s}
+\Big(\frac{c_2(\La_s)m_3(\La_s)}{m_1(\La_s)m_2(\La_s)}+b_2(\La_s)\Big)Y_s\Big)\d s\\
&\geq \liminf_{t\ra \infty}\frac 1t\int_0^t\!\Big(\!-\!a_2(\La_s)\!-\!\frac 12\beta^2(\La_s)
\!+\!\frac{c_2(\La_s)\varphi_s}{m_1(\La_s)\!+\!m_2(\La_s)\varphi_s}\Big)\d s\!-\!\limsup_{t\ra \infty}\frac 1t\ln Y_t, \quad a.s.
\end{split}
\end{equation}
Invoking (\ref{2.6.2}) and taking Proposition \ref{t-3} and Lemma
\ref{t-2} into consideration, we deduce that
\begin{equation}\label{2.6.5}
\begin{split}
&\liminf_{t\ra\infty} \frac 1t\int_0^t\!\Big(\frac{c_2(\La_s)(\varphi_s-X_s)}{m_1(\La_s)+m_2(\La_s)\varphi_s}
+\Big(\frac{c_2(\La_s)m_3(\La_s)}{m_1(\La_s)m_2(\La_s)}+b_2(\La_s)\Big)Y_s\Big)\d s\\
&\geq \lambda+\sum_{i\in \S}\mu_i\big(a_2(i)+\frac 12 \beta^2(i)-\frac{c_2(i)}{m_2(i)}\big)
+\sum_{i\in\S}\int_{\R_+}\!b_2(i)y\,\pi^{\psi}(\d y, i)=\bar\lambda,\quad a.s.
\end{split}
\end{equation}
and further that
\begin{equation}\label{2.6.6}
\liminf_{t\ra \infty}\frac 1t\int_0^t\!\Big(\frac{\check c_2}{\hat m_1}(\varphi_s-X_s)+\big(\frac{\check c_2\check m_3}{\hat m_1\hat m_2}+\check b_2\big)Y_s\Big)\d s
\geq \bar\lambda,\ a.s.
\end{equation}
Dividing both sides of (\ref{2.6.3}) and (\ref{2.6.6}) by $\hat b_1$ and $\check c_2/\hat m_1$ respectively, and adding them side by side, we obtain
\begin{equation}\label{2.6.7}
\liminf_{t\ra \infty}\frac 1t \int_0^tY_s\d s\geq \frac{\hat m_1}{\check c_2}\Big(\frac{\check m_3}{\hat m_2}+\frac{\check b_2\hat m_1}{\check c_2}+\frac{\check c_1}{\hat m_1\hat b_1}\Big)^{-1}\bar\lambda, \quad a.s.
\end{equation}
Hence, if $\bar\lambda>0$, then the inequality (\ref{2.6.7}) implies
that 
  $\limsup_{t\ra \infty} Y_t>0$ a.s.

Note the following facts:
\[\d \ln Y_t\leq \Big(-a_2(\La_t)-\frac 12\beta^2(\La_t)+\frac{c_2(\La_t) X_t}{m_1(\La_t)+m_2(\La_t)
X_t+m_3(\La_t)Y_t}\Big)\d t+\beta(\La_t)\d B_2(t),\] and $\lim_{t\ra
\infty}\frac 1t\int_0^t\beta(\La_s)\d B_2(s)=0$ a.s.  If
$\p(\{\omega: \lim_{t\ra\infty} X_t(\omega)=0\})>0$, then it must
hold
\[\p\Big(\limsup_{t\ra\infty} \frac 1 t\ln Y_t<0\Big)>0.\]
Hence, the fact that  $\limsup_{t\ra \infty} Y_t>0$ a.s. implies
that $\limsup_{t\ra\infty} X_t>0$ a.s.

At last, we show the existence of an invariant probability measure
of the process $(X_t,Y_t,\La_t)$. Note that almost surely  $X_t\leq
\varphi_t$ and $Y_t\leq \psi_t$ for any $t>0$. By (\ref{2.4}) and
(\ref{2.5}) of Lemma \ref{t-1}, we deduce  that for any $p>1$
satisfying $ -\hat a_2+\frac{\check c_2}{\hat
m_2}+\frac{p-1}{2}\check\beta^2>0$, there exists a constant $C>0$
such that
\[\frac1 t\int_0^t \E \big(X_s^p+Y_s^p+\La_s^p\big)\d s\leq C.\]
According to \cite[Theorem 4.14]{Chen}, there exists a stationary distribution for $(X_t,Y_t,\La_t)$. The proof is completed.
\end{proof}

\noindent\textbf{Proof of Theorem \ref{main}}\\
The argument follows immediately from 
Propositions \ref{t-2}, and \ref{t-4}-\ref{t-6}. \fin

\begin{rem}
  According to a suggestion of Hai Dang Nguyen, the condition (1.7) in Theorem 1.1 can be improved by noting that $\bar \lambda=\lambda$. We have showed that this equality holds when $N=1$. Actually it also holds for general $N>1$. 
  By It\^o's formula and the strong ergodic theorem, 
  \[\lim_{t\to \infty} \frac{\ln \psi(t)}{t}=\lim_{t\to \infty} \frac{\ln \psi(0)}{t}+\bar \lambda-\lambda=\bar \lambda -\lambda.\]
  Since $(\psi(t))$ is ergodic in the current situation, then $\lim_{t\to\infty} \frac{\ln\psi(t)}{t}=0$, which yields $\bar\lambda=\lambda$. 
\end{rem}

\vskip 0.3cm \noindent\textbf{Acknowledgment:} The authors are very grateful to the referees for their valuable comments. The first author is
supported by NSFC (No.11401592); The second author is supported by
NSFC (No.11301030), NNSFC (No.11431014), 985-project.


\begin{thebibliography}{99st}
%

\bibitem{AHS} L. Arnold, W. Horsthemke, J.W. Stucki, The influence of external real and white noise on the Lotka-Volterra model, {  Biom. J.},
{ 21} (1979), 451--471.

\bibitem{BBG} G.~K. Basak, A. Bisi, M.~K. Ghosh, Stability of a random diffusion
with linear drift,   {  J. Math. Anal. Appl.}, { 202} (1996),
604--622.

\bibitem{Chen} M.-F. Chen,  {From Markov Chains to Non-equilibrium Particle Systems}, World Scientific, 2nd edition, 2004.

\bibitem{CH} B. Cloez, M. Hairer, Exponential ergodicity for Markov processes with random switching, {  Bernoulli},  {  21}  (2015),    505--536.

\bibitem{DD} N.~H. Du, N.~H. Dang, Dynamics of Kolmogorov systems of
competitive type under the telegraph noise, {   J. Differential
Equations}, {\bf 250} (2011), 386--409.


\bibitem{DDY} N.~H. Du, N.~H. Dang, G. Yin, Conditions for permanence and ergodicity of certain stochastic predator-prey models. to appear in J. Appl. Prob.

\bibitem{Kha} R.~Z. Khasminskii,  Stochastic Stability of Differential Equations.
Nauka, Moscow (1969). English transl.: Sijthoff and Noordhoff,
Alphen, 1980


\bibitem{Sigm} J. Hofbauer, K. Sigmund,  {Evolutionary games and population dynamics}, Cambridge University Press, 1998.

\bibitem{IW} N. Ikeda,  S. Watanabe,   {Stochastic Differential Equations and
Diffusion Processes}, North-Holland, New York, 1989.

\bibitem{JJ} C. Ji, D. Jiang, Dynamics of a stochastic density dependent predator-prey system with Beddington-DeAngelis functional response,
{  J. Math. Anal. Appl.}, { 381} (2011), 441--453.

\bibitem{KK} R.~Z. Khasminskii, F.~C. Klebaner, Long term behavior of solutions of the Lotka-Volterra system under small random perturbations,
{  Ann. Appl. Probab.}, { 11} (2001), 952-963.


\bibitem{LW} M. Liu, K. Wang, Global stability of a nonlinear
stochastic predator-prey system with Beddington-DeAngelis functional
response, {  Commun. Nonlinear Sci. Numer. Simulat.}, { 16}
(2011), 1114--1121.

\bibitem{LWa} M. Liu, K. Wang, The threshold between permanence and
extinction for a stochastic Logistic model with regime switching,
{  J. Appl. Math. Comput.}, { 43} (2013), 329--349.



\bibitem{LZ} S. Li,   X. Zhang, Dynamics of a stochastic non-autonomous
predator-prey system with Beddington-DeAngelis functional response,
{  Adv. Difference Equ.}, 2013, 2013:19.


\bibitem{LAJM} X. Li, A. Gray, D. Jiang, X. Mao,   Sufficient and necessary conditions of stochastic permanence and extinction for stochastic logistic
 populations under regime switching,
 {  J. Math. Anal. Appl.},  { 376}  (2011),    11--28.

\bibitem{Mao} X. Mao,  {Stochastic differential equations and applications},  Horwood Publishing Limited, 1997.

\bibitem{MY} X. Mao, C. Yuan,  {Stochastic Differential Equations with Markovian Switching}, Imperial College Press, London, 2006.

\bibitem{PP} M. Pinsky, R. Pinsky,  {Transience recurrence and central limit theorem behavior for diffusions in random temporal environments}, Ann. Probab. {  21} (1993), 433--452.

\bibitem{PS} R. Pinsky, M. Scheutzow, Some remarks and examples concerning the transience and recurrence of random diffusions.
Ann. Inst. Henri. Poincar\'e , 28  (1992), 519--536.

\bibitem{Sh14a} J. Shao, Ergodicity  of  regime-switching diffusions in Wasserstein distances,  Stoch. Proc. Appl., 125 (2015),   739--758.

\bibitem{Sh14b} J. Shao, Criteria for transience and recurrence of regime-switching diffusion processes,   Electron. J. Probab., 20 (2015),  1--15.

\bibitem{Sh1} J. Shao, Ergodicity of one-dimensional regime-switching diffusion processes, {  Science China Math.}, { 57} (2014), 2407--2414.

\bibitem{SX} J. Shao, F.~B. Xi, Strong ergodicity of the regime-switching diffusion processes,  Stoch. Proc. Appl.,    123  (2013), 3903--3918.

\bibitem{SX2} J. Shao, F.~B. Xi, Stability and recurrence of regime-switching diffusion processes, SIAM J. Control Optim.,  52 (2014), 3496--3516.


\bibitem{S89} A.~V.  Skorokhod,  {Asymptotic methods in the theory of
stochastic differential equations}, Vol. 78, American Mathematical
Soc., 1989.

\bibitem{YZ} G.~G. Yin, C. Zhu,  {Hybrid switching diffusions: properties and applications}, Vol. 63, Stochastic Modeling and Applied Probability,
Springer, New York, 2010.

 \bibitem{YML} C. Yuan, X. Mao, J. Lygeros,   Stochastic hybrid delay population dynamics: well-posed models and
 extinction,
  { J. Biol. Dyn.},  { 3}  (2009),   1--21.

 \bibitem{ZYb}C. Zhu, G. Yin,  On competitive Lotka-Volterra model in random
environments, {  J. Math. Anal. Appl.},  { 357}  (2009),
154--170.


\bibitem{ZY09}C. Zhu, G. Yin,   On hybrid competitive Lotka-Volterra ecosystems, {  Nonlinear Anal.},  { 71}  (2009),
  e1370--e1379.




\end{thebibliography}
\end{document}